\documentclass[a4paper, 12pt]{amsart}
\usepackage{amsmath}
\usepackage{amscd}
\usepackage{amssymb}
\usepackage{amsthm}
\usepackage{amsfonts}
\usepackage{latexsym}
\usepackage{verbatim}
\usepackage{amsthm}

\usepackage[left=2cm, top=2.6cm, bottom=2.6cm, right=2cm]{geometry}

\numberwithin{equation}{section}

\theoremstyle{plain}

\newcommand{\id}{\mathop{\mathrm{id}}\nolimits}

\newcommand{\J}{\mathop{\frak J}\nolimits}
\newcommand{\wJ}{\mathop{\widetilde{\frak J}}\nolimits}

\newcommand{\R}{\mathbb{R}}
\newcommand{\C}{\mathbb{C}}

\newtheorem{newstatement}{newstatement}
\newtheorem{definition}[newstatement]{Definition}
\newtheorem{lemma}[newstatement]{Lemma}
\newtheorem{theorem}[newstatement]{Theorem}
\newtheorem{corollary}[newstatement]{Corollary}
\newtheorem{question}[newstatement]{Question}

\newtheorem{proposition}[newstatement]{Proposition}

\begin{document}

\title[Embedding almost-complex manifolds]{Embedding almost-complex manifolds in almost-complex Euclidean spaces}
\author{Antonio J. Di Scala \and Daniele Zuddas} \date{\today} \address{Dipartimento di Matematica \\ Politecnico di Torino \\ Corso Duca degli Abruzzi 24, 10129 Torino, Italy} \email{antonio.discala@polito.it} \address{Dipartimento di Matematica e Informatica \\ Universit\`a di Cagliari\\ Via Ospedale 72 \\ 09124 Cagliari\\ Italy} \email{d.zuddas@gmail.com}


\thanks{This work was supported by the Project M.I.U.R. PRIN~05 ``Riemannian Metrics and  Differentiable Manifolds'' and by G.N.S.A.G.A. of I.N.d.A.M.\\
\indent The second author was supported by Regione Autonoma della Sardegna with funds from PO Sardegna 
FSE~2007--2013 and L.R.~7/2007 ``Promotion of scientific research and technological 
innovation in Sardinia''.}

\begin{abstract} We show that any compact almost-complex manifold $(M,J)$ of complex dimension $m$ can be pseudo-holomorphically embedded in $\R^{6m}$ equipped with a suitable almost-complex structure $\widetilde J$.


\end{abstract}
\maketitle

\section{Introduction} An almost-complex structure on a $2n$-dimensional smooth manifold $M$ is a tensor $J\in {\rm End}(TM)$ such that $J^2= -\id$. If $M$ is oriented we say that $J$ is \emph{positive} if the orientation induced by $J$ on $M$ agrees with the given one. An almost-complex structure is called \emph{integrable} if it is induced by a holomorphic atlas.
In dimension two any almost-complex structure is integrable, while
in higher dimension this is far from true.
A smooth map $f\colon N\to M$ between two almost-complex manifolds $(N,J')$, $(M,J)$ is called {\em pseudo-holomorphic} if $J \circ Tf = Tf \circ J'$, where $Tf : TN \to TM$ is the tangent map of $f$. When the map $f$ is an embedding, $(N,J')$ is said to be an \emph{almost-complex submanifold} of $(M,J)$. In this case we can identify $N$ with its image $f(N) \subset M$ and the almost-complex structure $J'$ with the restriction of $J$ to $TN \cong T(f(N)) \subset TM$.

If we equip $\R^{2n}$ with the canonical complex structure, that is to say $\R^{2n} \cong \C^n $, then it does not admit any compact complex submanifold (by the maximum principle). Thus, it is a very natural problem to ascertain if it is possible to find compact complex manifolds pseudo-holomorphically embedded in $\R^{2n}$ equipped with an integrable or non-integrable almost-complex structure.

In \cite{CE53} Calabi and Eckmann constructed the first examples of compact, simply connected complex manifolds $M_{p,q}$ which are not algebraic. Topologically $M_{p,q}$ is the product $S^{2p+1} \times S^{2q+1}$. Then by deleting a point on each factor one obtains a complex structure $J$ on $\mathbb{R}^{2p + 2q + 2}$. In section 5 of \cite{CE53} it was shown that when $p,q  > 1$ there exists a complex torus as a complex submanifold of $(\mathbb{R}^{2p + 2q + 2}, J)$ \cite[p.\,499]{CE53}. It follows that the Calabi-Eckmann complex structure $J$ on $\mathbb{R}^{2n}$ cannot be tamed by any symplectic form and in particular cannot be K\"ahler.
Calabi and Eckmann also observed that the only holomorphic functions on $(\mathbb{R}^{2p + 2q + 2}, J)$ are the constants answering negatively to a question raised by Bochner about the uniformization of complex structures on $\mathbb{R}^{2n}$. In \cite{Bry82} Bryant constructed  pseudo-holomorphic non-constant maps $\varphi: M^2 \rightarrow S^6$ for any compact Riemann surface $M^2$, where $S^6$ is equipped with the almost-complex structure induced by the octonion multiplication. These maps realize compact Riemann surfaces as pseudo-holomorphic singular curves in $S^6$.

In \cite{DV09} it was shown that any almost-complex torus $\mathbb{T}^n=\R^{2n}/\Lambda$ can be pseudo-hol\-o\-mor\-phically embedded into $(\R^{4n}, J_{\Lambda})$ for a suitable almost-complex structure $J_\Lambda$. It follows that any compact Riemann surface can be realized as a pseudo-holomorphic curve of some $(\mathbb{R}^{2n},J)$, where $J$ is a suitable almost-complex structure.

In this paper we prove the following general theorem.

\begin{theorem}\label{main}
Any compact almost-complex manifold $(M,J)$ of real dimension $2m$ can be pseudo-holo\-mor\-phically embedded in $(\R^{6m},\widetilde{J})$ for a suitable positive almost-complex structure $\widetilde J$.
\end{theorem}

In particular, any compact Riemann surface can be realized as a pseudo-holomorphic curve in $(\R^6,\widetilde J)$.
In \cite{DV09} was shown that the torus is the only compact Riemann surface that can be pseudo-holomorphically embedded in $(\R^{4}, \widetilde J)$ for some $\widetilde J$.

\section{Preliminaries}

The space of positive linear complex structures on $\R^{2n}$ is diffeomorphic to the homogeneous space $\wJ(n) =
GL^+(2n, \R) / GL(n, \C)$ and is homotopy equivalent to $\J(n) = SO(2n) / U(n)$.
So, an almost-complex structure $J$ on $\R^{2n}$ can be regarded as a smooth map $J : \R^{2n} \to \wJ(n)$.

\begin{lemma} \label{HEP} Let $M\subset \R^{2n}$ be a closed submanifold and let $J : M \to \wJ(n)$ be a smooth map. Then there exists a smooth extension $\widetilde J : \R^{2n} \to \wJ(n)$ if and only if $J$ is homotopic to a constant.
\end{lemma}

\begin{proof} The `only if' part follows immediately from the fact that $\R^{2n}$ is contractible.

Let us prove the `if' part. Consider a smooth homotopy
$H : M \times [0,1] \to \wJ(2n)$ such that $H_0(x) = J_0$ for all $x \in M$, and $H_1 = J$ where $H_t(x) = H(x,t)$ and $J_0 \in \wJ(n)$. We can extend $H$ to $\R^{2n} \times \{0\} \subset \R^{2n} \times [0,1]$ by setting $H(x,0) = J_0$ for any $x \in \R^{2n}$. By the homotopy extension property \cite[Chapter 0]{hatcher} there exists $\widetilde H : \R^{2n} \times [0,1] \to \wJ(n)$ which extends $H$. We conclude the proof by setting $\widetilde J = \widetilde H_1$.
\end{proof}

Let $(M,J)$ be an almost-complex manifold. The strategy to prove Theorem \ref{main} will be to choose an arbitrary embedding $f:M \hookrightarrow \R^{6m}$, which exists for the weak Whitney embedding theorem, and to show that $J$ extends to the pullback $f^*(T\R^{6m})$ and this extension is null-homotopic.

Consider the standard filtration $SO(1)\subset SO(2) \subset \cdots$. Since $SO(n-1)$ contains the $(n-2)$-skeleton of $SO(n)$ (because the standard fibration $SO(n) \to S^{n-1}$) it follows that the $k$-skeleton of $SO(n)$ is contained on $SO(k+1)$ for $0 \leq k \leq n-2$.

Since $SO(n) \subset U(n)$ it follows that $U(n)$ contains the $(n-1)$-skeleton of $SO(2n)$ for $n \geq 1$. Then the homomorphism induced by the inclusion $i_* : \pi_j(U(n)) \to \pi_j(SO(2n))$ is an isomorphism for $j \leq n-2$ and is an epimorphism for $j = n-1$.

From the homotopy exact sequence of the fibre bundle $SO(2n) \to \J(n)$ given by the projection map it follows that $\pi_j(\wJ(n))\cong \pi_j(\J(n)) \cong 0$ for $j \leq n-1$.

\begin{definition}
A space $X$ is said to be $n$-connected if $\pi_j(X) \cong 0$ for all $j \leq n$.
\end{definition}

In particular, $0$-connected means path-connected.

From the above considerations we have that $\wJ(n)$ is $(n-1)$-connected. The following proposition is well-known in the theory of CW-complexes.

\begin{proposition}\label{null-homot}
If $X$ is $n$-connected then any map $Y \to X$ defined on a CW-complex $Y$ of dimension $ \leq n$ is homotopic to a constant.
\end{proposition}

Also the following proposition is standard, and we give only the idea of the proof.

\begin{proposition}\label{sections/pro}
Let $\xi : E \to M$ be an oriented real vector bundle of rank $2k$ over an $m$-manifold $M$. If $k \geq m$ then $\xi$ admits a positive complex structure.
\end{proposition}

\begin{proof}
Consider the bundle $\xi^{\J} : \wJ(E) \to M$ with fibre $\wJ(k)$ induced by $\xi$. Namely, for any $p\in M$ the fibre of $\xi^{\J}$ over $p$ is the space of positive linear complex structures on $\xi^{-1}(p)$. Since $\wJ(k)$ is $(k-1)$-connected, it follows that $\xi^{\J}$ admits a section if $k \geq m$, see \cite[Part~III]{steenrod}. This section is a positive complex structure on $\xi$.
\end{proof}

Let $f:M \rightarrow \mathbb{R}^N$ be an immersion. The normal bundle $\nu_f(M)$ is, as usual, the orthogonal complement of $TM$ in $f^*(T \mathbb{R}^N)$, that is to say: \[ f^*(T \mathbb{R}^N) = TM \oplus \nu_f(M). \]

If $M$ is oriented then the normal bundle can be equipped with a canonical orientation, namely that which makes the splitting of $f^*(T\R^N)$ into a Whitney sum of oriented fibre bundles, where $\R^N$ is considered with the standard orientation.

\section{Proof of the main results}

\begin{theorem}\label{k-ext}
Let $M \subset \R^{2n}$ be a submanifold of even dimension endowed with an almost-complex structure $J$. If the normal bundle of $M$ in $\R^{2n}$ admits a positive complex structure with respect to the canonical orientation, then for any $k \geq \max(0, \dim_{\R} M - n + 1)$ there exists an almost-complex structure $\widetilde J$ on $\R^{2n} \times \R^{2k}$ such that $M \times \{0\} \subset \R^{2n} \times \R^{2k}$ is an almost-complex submanifold.
\end{theorem}

\begin{proof}
Let us choose a positive complex structure on the normal bundle of $M$. Then by taking the Whitney sum with the almost-complex structure on $M$ we get a complex structure on $(T\R^{2n})_{|M}$. So we obtain a smooth map $J : M \to \wJ(n)$.

In view of Lemma \ref{HEP} our target is to get a $J$ null-homotopic. This is so if $\dim_{\R} M \leq n-1$ because $\wJ(n)$ is $(n-1)$-connected and Proposition~\ref{null-homot}.

If $\dim_{\R} M > n-1$ we take the product $\R^{2n} \times \R^{2k}$, where $\R^{2k}$ is endowed with the standard complex structure, and we embed $M$ as $M \times \{0\}$. We get a complex structure on the normal bundle of $M$ in $\R^{2n} \times \R^{2k}$ in the obvious way. So we obtain a map $J_k : M \to \wJ(n+k)$. It follows that $J_k$ is homotopic to a constant if $k \geq \dim_{\R} M - n +1$. In this case $J_k$ extends on $\R^{2n} \times \R^{2k}$ by Lemma \ref{HEP}.
\end{proof}

It follows that if $(M,J)$ is contained in $\C^n$ with a complex normal bundle and if $n \geq 2 \dim_{\C} M + 1$, then there is a positive almost-complex structure $\widetilde J$ on $\C^n$ which makes $(M, J)$ an almost-complex submanifold of $(\C^n, \widetilde J)$.

\begin{proof}[Proof of Theorem \ref{main}]
Let $f : M \hookrightarrow \R^{6m}$ be any embedding. The normal bundle $\nu_f(M)$ has rank $4m$ and is orientable. By Proposition \ref{sections/pro} there is a complex structure on the normal bundle and then we conclude by an application of Theorem \ref{k-ext} with $k = 0$.
\end{proof}

In some cases we can construct an embedding in an euclidean space of lower dimension.
Recall that an \emph{s-inverse} of the tangent bundle $TM$ is a vector
bundle $\xi$ such that $TM \oplus \xi$ is a trivial vector bundle. Observe that if $f: M \rightarrow \mathbb{R}^{N}$ is an immersion then the normal bundle $\nu_f(M)$ is a real s-inverse of the tangent bundle $TM$. The converse also holds and is a Theorem of Hirsch \cite{Hir59}, and is given as follows.

\begin{theorem} \label{Hirsh} (Hirsch \cite{Hir59})
Any s-inverse of $TM$ is the normal bundle of some immersion $f: M \rightarrow \mathbb{R}^{N}$.
\end{theorem}

Let $\xi$ be a complex s-inverse of $(TM,J)$ of complex rank $k$, namely $TM \oplus \xi$ is trivial as a real vector bundle. Now Hirsch's Theorem \ref{Hirsh} implies that there exists an immersion $f:M \rightarrow \R^{2(m+k)}$ such that $\xi$ is isomorphic to $\nu_f(M)$ as real vector bundles. So $\nu_f(M)$ carries a complex structure.

Up to a product with some $\R^{2h}$, we can assume that $k \geq m + 1$, and then $f$ is regularly homotopic, namely homotopic through immersions, to an embedding $f_1 : M \to \R^{2(m+k)}$. It follows that $\nu_{f_1}(M) \cong \nu_f(M)$ carries a complex structure.
Now apply Theorem \ref{k-ext} to get $\widetilde J$.

If the rank of $\xi$ satisfies $m + 1 \leq k \leq 2m - 1$ we get a pseudo-holomorphic embedding in an euclidean space of complex dimension $m +k < 3m$.

Let $(S^6,J)$ be the six-dimensional sphere equipped with the standard almost-complex structure $J$ obtained from the octonion multiplication.
Theorem \ref{main} implies that $(S^6,J)$ can be pseudo-holo\-mor\-phically embedded in $(\R^{18},\widetilde{J})$ for a suitable positive almost-complex structure $\widetilde J$. Using the existence of a low-dimensional s-inverse of $(TS^6,J)$ we have the following result.

\begin{corollary}
The almost-complex sphere $(S^6,J)$ can be pseudo-holo\-mor\-phically embedded in $(\R^{14},\linebreak[1] \widetilde{J})$ for a suitable positive almost-complex structure $\widetilde J$.
\end{corollary}
\begin{proof}
Since $S^6$ is embedded in $\mathbb{R}^8$ with trivial normal bundle we conclude by an application of Theorem \ref{k-ext} with $k = 3$.
\end{proof}

Notice that $(S^6,J)$ can not be pseudo-holo\-mor\-phically embedded in $(\R^{12},\widetilde{J})$. In fact, the Euler class of the normal bundle of any embedding of $S^6$ in $\R^{12}$ is zero by a theorem of Whitney, see \cite[p. 138]{Hir76}. On the other hand, if $S^6$ is contained pseudo-holomorphically in $(\R^{12}, \widetilde J)$, by a straightforward computation with the Chern class, we obtain for the Euler class $e(\nu(S^6)) = c_3(\nu(S^6)) = -2 \lambda \neq 0$, which is a contradiction, where $\lambda \in H^6(S^6)$ is the standard generator.

We conclude with a question. Since our construction is essentially homotopy-the\-o\-ret\-ic, we are unable to control the integrability of the almost-complex structure $\widetilde J$ of Theorem~\ref{main}. So the following question is very natural.

\begin{question}
Let $(M, J)$ be an integrable complex manifold. Is there an embedding of $(M, J)$ into an integrable $(\R^{2n}, \widetilde J)$?
\end{question}

\vspace{0.5cm}
\begin{center}
{\bf Acknowledgment}
\end{center}
We thank Simon Salamon for useful comments and remarks.

\end{document}